\title{\textbf{Ruin probability in a risk model with \\a variable premium intensity and risky investments}}
\author{
        Yuliya~Mishura \thanks{Department of Probability Theory, Statistics and
                               Actuarial Mathematics, Taras Shevchenko National University
                               of Kyiv, 64 Volodymyrska, 01601 Kyiv, Ukraine,
                               e-mail: myus@univ.kiev.ua}
        \and
        Mykola~Perestyuk \thanks{Department of Integral and Differential
                                 Equations, Taras Shevchenko National University
                                 of Kyiv, 64 Volodymyrska, 01601 Kyiv, Ukraine,
                                 e-mail: pmo@univ.kiev.ua}
        \and
        Olena~Ragulina \thanks{Department of Probability Theory, Statistics and
                               Actuarial Mathematics, Taras Shevchenko National University
                               of Kyiv, 64 Volodymyrska, 01601 Kyiv, Ukraine,
                               e-mail: lena\_ragulina@mail.ru}
       }
\date{\today}

\documentclass[12pt]{article}
\usepackage{amsmath,amssymb,amsthm}
\usepackage[english]{babel}
\usepackage[cp1251]{inputenc}

\theoremstyle{plain}
\newtheorem{theorem}{Theorem}
\newtheorem{proposition}{Proposition}
\newtheorem{lemma}{Lemma}
\theoremstyle{remark}
\newtheorem{example}{Example}
\newtheorem{remark}{Remark}

\begin{document}
\maketitle

\begin{abstract}
We consider a generalization of the classical risk model when
the premium intensity depends on the current surplus of an insurance company.
All surplus is invested in the risky asset, the price of which follows
a geometric Brownian motion. We get an exponential bound for
the infinite-horizon ruin probability. To this end, we allow
the surplus process to explode and investigate the question
concerning the probability of explosion of the surplus process
between claim arrivals.
\end{abstract}

\begin{quote}

\end{quote}

\section{Introduction}\label{se:1}

Since Lundberg introduced the collective risk model in 1903,
the estimation of the ruin probability has been one of the central directions
for investigations in risk theory.
It is well known that in the Cram\'er-Lundberg model, which is also
called the classical risk model, the infinite-horizon ruin probability
decreases exponentially with the initial surplus if the claim sizes have exponential
moments and the net profit condition holds.
Results concerning bounds and asymptotics for the ruin probability were also obtained
for different generalizations of the classical risk model under various assumptions
(see, e.g.,~\cite{no2, no7, no20} and the references given there).

Risk models that allow the insurance company to invest are of great interest.
The fact that risky investments can be dangerous was first justified mathematically
by Kalashnikov and Norberg~\cite{no12}. They modelled the basic surplus process
due to insurance activity and the price of the risky asset by L\'evy processes
and obtained upper and lower power bounds for the ruin probability
when the initial surplus is large enough.
Later, Paulsen~\cite{no18} and Yuen, Wang, Wu~\cite{no22} considered some
generalizations of these results.

Frolova, Kabanov and Pergamenshchikov~\cite{no5} used the bounds obtained in~\cite{no12}
to show that the ruin occurs with probability~1 in the classical risk model
if all surplus is invested in the risky asset, the price of which is modelled
by a geometric Brownian motion, and some additional conditions for parameters
of the geometric Brownian motion hold. They also showed that if these conditions
are not fulfilled, a power asymptotic is true for the ruin probability
when the claim sizes are exponentially distributed.
The power asymptotic was got by Cai and Xu~\cite{no4} in the case where the classical
risk process is perturbed by a Brownian motion.
Moreover, Pergamenshchikov and Zeitouny~\cite{no19} considered the risk model where
the premium intensity is a bounded nonnegative random function and
generalized results of~\cite{no5}.

On the other hand, numerous results indicate that risky investments can be used
to improve the solvency of the insurance company. For example,
Gaier, Grandits and Schachermayer~\cite{no6} considered the classical risk model
under the additional assumptions that the company is allowed to borrow and invest
in the risky asset, the price of which follows a geometric Brownian motion.
They obtained an upper exponential bound for the ruin probability when
the claim sizes have exponential moments and a fixed quantity, which is independent
of the current surplus, is invested in the risky asset. It appears that this bound
is better then the classical one. For an exponential bound in a model
with risky investments see also, for instance,~\cite{no16}.

Numerous investigations are devoted to solving optimal investment problems
from viewpoint of the infinite-horizon ruin probability minimization.
For instance, Hipp and Plum~\cite{no9}, Liu and Yang~\cite{no15},
Azcue and Muler~\cite{no3} considered the optimal investment problem
in the classical risk model when the company is allowed to borrow.
Asymptotics for the ruin probability under optimal strategies
were obtained by Hipp and Schmidli~\cite{no10}, Grandits~\cite{no8},
Schmidli~\cite{no21} for different assumptions about claim sizes.

We consider a generalization of the classical risk model when the premium intensity
depends on the current surplus of the insurance company, which is invested in the risky asset.
Our main aim is to show that if the premium intensity grows rapidly with
increasing surplus, then an exponential bound for the ruin probability
holds under certain conditions in spite of the fact that all surplus
is invested in the risky asset.
To this end, we allow the surplus process to explode.
To be more precise we let the premium intensity be a quadratic function.
In addition, we investigate the question concerning the probability of explosion
of the surplus process between claim arrivals in detail.

Let \((\Omega, \mathfrak{F}, \mathbb{P})\) be a probability space satisfying
the usual conditions and all the objects be defined on it.
We assume that the insurance company has a nonnegative initial surplus \(x\)
and denote by \(X_t (x)\) its surplus at time \(t\ge 0\). For simplicity of notation,
we write \(X_t\) instead of \(X_t (x)\) when no confusion can arise.
Let \(c\colon \mathbb{R_+} \to \mathbb{R_+} \!\setminus\! \{0\}\) be
a measurable function such that \(c(u)= c(0)\) for all \(u<0\)
and \(c(X_t)\) be a premium intensity that depends on the surplus at time \(t\).

Next, we suppose that the claim sizes form a sequence \((Y_i)_{i\ge 1}\) of
nonnegative i.i.d. random variables with finite expectations \(\mu\).
We denote by \(\tau_i\) the time when the \(i\)th claim arrives. For convenience we set \(\tau_0 =0\).

Let \(h\colon \mathbb{R_+} \to \mathbb{R_+}\) be the shifted moment generating function
of \(Y_i\) such that \(h(0)=0\), i.e.
\[
h(r)=\mathbb{E} e^{rY_i} -1.
\]
We make the following classical assumption concerning \(h(r)\):
there exists \(r_{\infty} \in (0,+\infty]\) such that \(h(r)< +\infty\) for all \(r \in [0,r_{+\infty})\)
and \( \lim_{r\uparrow r_{\infty}} h(r)= +\infty\) (see \cite[p.~2]{no7}).
It is easily seen that \(h(r)\) is increasing, concave, and continuous on \([0,r_{+\infty})\).

The number of claims on the time interval \([0,t]\) is a Poisson process \((N_t)_{t\ge 0}\)
with constant intensity \(\lambda >0\).
Thus, the total claims on \([0,t]\) equal \(\sum_{i=1}^{N_t} Y_i\).
We set \(\sum_{i=1}^{0} Y_i =0\) if \(N_t =0\).

In addition, we assume that all surplus is invested in the risky asset, the price of which
equals \(S_t\) at time \(t\). We model the process \((S_t)_{t\ge 0}\) by a geometric
Brownian motion. Thus,
\begin{equation}
\label{eq:1.1}
dS_t= S_t (a\,dt +b\,dW_t),
\end{equation}
where \(a>0\), \(b>0\), and \((W_t)_{t\ge 0}\) is a standard Brownian motion.
We suppose that the random variables \((Y_i)_{i\ge 1}\) and the processes \((N_t)_{t\ge 0}\)
and \((W_t)_{t\ge 0}\) are independent.

Let \((\mathfrak{F}_t)_{t\ge 0}\) be a filtration generated by \((Y_i)_{i\ge 1}\), \((N_t)_{t\ge 0}\),
and \((W_t)_{t\ge 0}\), i.e.
\[
\mathfrak{F}_t= \sigma \bigl((N_s)_{0\le s\le t}, (W_s)_{0\le s\le t}, Y_1, Y_2, \dots, Y_{N_t}\bigr).
\]

Under the above assumptions, the surplus process \((X_t)_{t\ge 0}\) follows the equation
\begin{equation}
\label{eq:1.2}
X_t= x+ \int_0^t c(X_s)\,ds+ \int_0^t \frac{X_s}{S_s} \,dS_s- \sum_{i=1}^{N_t} Y_i, \quad t\ge 0.
\end{equation}

Substituting~\eqref{eq:1.1} into~\eqref{eq:1.2} yields
\begin{equation}
\label{eq:1.3}
X_t= x+ \int_0^t c(X_s)\,ds+ a\int_0^t X_s \,ds+ b\int_0^t X_s \,dW_s- \sum_{i=1}^{N_t} Y_i, \quad t\ge 0.
\end{equation}

The ruin time is defined as \(\tau(x)= \inf\{t\ge0\colon X_t(x) <0\}\).
We suppose that \(\tau(x) =\infty\) if \(X_t (x) \ge 0\) for all \(t\ge 0\).
To simplify notation, we let \(\tau\) stand for \(\tau(x)\).
The corresponding infinite-horizon ruin probability is given by
\(\psi(x)= \mathbb{P} \bigl[ \inf_{t\ge0} X_t(x) <0 \bigr]\), which is equivalent to
\(\psi(x)= \mathbb{P} [\tau(x)< \infty]\).

The rest of the paper is organized in the following way.
Section~\ref{se:2} deals with the detailed investigation of the question
concerning the probability of explosion of the risk process between claim arrivals.
In Section~\ref{se:3} we formulate and prove the existence and uniqueness theorem
for stochastic differential equation that describes the surplus process.
In Section~\ref{se:4} we establish the supermartingale property
for an auxiliary exponential process. This property allows us to get an exponential bound
for the ruin probability under certain conditions.
Finally, in Section~\ref{se:5} we consider the case where the premium intensity is
a quadratic function and obtain an exponential bound for the ruin probability.
In addition, Appendix~\ref{se:A} gives two lemmas, which are used in Section~\ref{se:2}.

\section{Auxiliary results}\label{se:2}

Consider now the following stochastic differential equation
\begin{equation}
\label{eq:2.1}
X_t= x+ \int_0^t p(X_s)\,ds+ b\int_0^t X_s\,dW_s, \quad t\ge0,
\end{equation}
where \(x>0\), \(b>0\), \((W_t)_{t\ge0}\) is a standard Brownian motion,
\(p\colon \mathbb{R} \to \mathbb{R_+}\)
is a locally Lipschitz continuous function such that \(p(u)\)
is strictly increasing on \(\mathbb{R_+}\) and \(p(u)=p(0)\) for all \(u<0\).

Equation~\eqref{eq:2.1} describes the surplus process between two successive jumps
of \((N_t)_{t\ge 0}\) in the model considered above provided that one
puts the corresponding restrictions on \(c(u)\), sets \(p(u)=c(u)+au\), and
takes the surplus at time when the last jump of \((N_t)_{t\ge 0}\) occurs instead of \(x\).

First, we give some results which show that \((X_t)_{t\ge0}\) goes to \(+\infty\)
either with probability 1 or with positive probability, which is less then 1 under certain conditions.

Let \(t^*\) be a possible explosion time of \((X_t)_{t\ge0}\), i.e. \(t^*= \inf\{t\ge0\colon X_t=\infty\}\).
Moreover, we denote by \(t_{(0,+\infty)}^*\) the first exit time from \((0,+\infty)\) for \((X_t)_{t\ge0}\), i.e.
\(t_{(0,+\infty)}^*= \inf\{t\ge0\colon X_t\notin (0,+\infty)\}\).
By Theorem~3.1 in~\cite[p.~178--179]{no11}, equation~\eqref{eq:2.1} has a unique strong solution up to
the explosion time \(t^*\). Note that here and subsequently,
we imply the pathwise uniqueness of solutions only.

For \(x>0\), we define
\begin{equation}
\label{eq:2.2}
I_1= \int_x^{+\infty} \exp\left\{ -\frac{2}{b^2} \int_x^v \frac{p(u)}{u^2}\,du \right\} dv
\quad\text{and}\quad
I_2= -\int_0^x \exp\left\{ \frac{2}{b^2} \int_v^x \frac{p(u)}{u^2}\,du \right\} dv.
\end{equation}

\begin{proposition}
\label{pr:2.1}
If \(p(0)>0\) and condition~\eqref{eq:A.1} holds, then
\[
\mathbb{P} \bigl[ \lim\nolimits_{t\to t_{(0,+\infty)}^*} X_t= +\infty \bigr] =1.
\]
\end{proposition}

\begin{proof}
Note that in this case \(I_1 < +\infty\) and \(I_2 = -\infty\) by Lemmas~\ref{lem:A.1} and~\ref{lem:A.2}.
Thus, the assertion of the proposition follows immediately from Theorem~3.1 in~\cite[p.~447]{no11}.
\end{proof}

\begin{remark}
\label{rem:2.2}
If \(p(0)=0\), then \(I_2\) may be finite. By Theorem~3.1 in~\cite[p.~447]{no11},
if \(I_1<+\infty\) and \(I_2>-\infty\),
then \( \lim_{t\to t_{(0,+\infty)}^*} X_t\) exists a.s.,
\(0< \mathbb{P} \bigl[ \lim_{t\to t_{(0,+\infty)}^*} X_t= +\infty \bigr] <1 \),
and \(\mathbb{P} \bigl[ \lim_{t\to t_{(0,+\infty)}^*} X_t= 0 \bigr]
=1- \mathbb{P} \bigl[ \lim_{t\to t_{(0,+\infty)}^*} X_t= +\infty \bigr] \).
\end{remark}

\begin{remark}
\label{rem:2.3}
Proposition~\ref{pr:2.1} does not give us whether the exit time \(t_{(0,+\infty)}^*\) is finite.
It is well known that Feller's test for explosions
(see, e.g., Theorem~5.29 in \cite[p.~348]{no13} and~\cite{no14})
gives precise conditions for whether or not a one-dimensional diffusion process
explodes in finite time. This test is very useful when one wants to show that
a diffusion process does not explode in finite time (see, e.g.,~\cite{no17}),
but it does not solve our problem.
\end{remark}

We now give a few examples.

\begin{example}
\label{ex:2.4}
Let
\[
p(u)=
\begin{cases}
p_1 u +p_0 &\text{if} \quad u\ge0, \\
p_0 &\text{if} \quad u<0.
\end{cases}
\]
The function \(p(u)\) has the asserted properties provided that \(p_0 \ge0\) and \(p_1 >0\).

Since
\[
I_1= \int_x^{+\infty} \exp \left\{ -\frac{2}{b^2}
\int_x^v \frac{p_1 u+ p_0}{u^2} \,du \right\} dv
=\int_x^{+\infty} \left( \frac{x}{v} \right)^{2p_1/b^2} \!\!\cdot
\exp \left\{ \frac{2p_0}{b^2} \left( \frac{1}{v}-\frac{1}{x} \right) \right\} dv,
\]
we have \(I_1= +\infty\) for \(2p_1 \le b^2\), and \(I_1< +\infty\) for \(2p_1> b^2\).

We first consider the case \(p_0 >0\).
From Theorem~3.1 in~\cite[p.~447]{no11} and Lemma~\ref{lem:A.2} we conclude that
\(\mathbb{P} \bigl[ t_{(0,+\infty)}^*= \infty \bigr] =1 \) if \(2p_1 \le b^2\), and
\(\mathbb{P} \bigl[ \lim_{t\to t_{(0,+\infty)}^*} X_t= +\infty \bigr] =1 \) if \(2p_1> b^2\).

Consider now the case \(p_0 =0\). Since
\[
I_2= -\int_0^x \exp\left\{ \frac{2}{b^2} \int_v^x \frac{p_1 u}{u^2}\,du \right\} dv
=-\int_0^x \left(\frac{x}{v}\right)^{2p_1/b^2} dv,
\]
we get \(I_2> -\infty\) for \(2p_1 < b^2\), and \(I_2= -\infty\) for \(2p_1 \ge b^2\).
Theorem~3.1 in~\cite[p.~447]{no11} yields
\(\mathbb{P} \bigl[ \lim_{t\to t_{(0,+\infty)}^*} X_t= 0 \bigr] =1 \) if \(2p_1 < b^2\),
\(\mathbb{P} \bigl[ t_{(0,+\infty)}^*= \infty \bigr] =1 \) if \(2p_1 = b^2\), and
\(\mathbb{P} \bigl[ \lim_{t\to t_{(0,+\infty)}^*} X_t= +\infty \bigr] =1 \) if \(2p_1> b^2\).
\end{example}

\begin{example}
\label{ex:2.5}
Let
\[
p(u)=
\begin{cases}
p_1 (u+p_2)^{\alpha} &\text{if} \quad u\ge0, \\
p_1 p_2^{\alpha} &\text{if} \quad u<0.
\end{cases}
\]
We put the following restrictions on the parameters of \(p(u)\):
\(\alpha >1\), \(p_1 >0\), and \(p_2 \ge0\).

Since
\begin{equation*}
\begin{split}
&\lim_{v\to +\infty} \left( (1+\varepsilon) \ln v
-\frac{2}{b^2} \int_x^v \frac{p_1 (u+p_2)^{\alpha}}{u^2}\,du \right)
\le \lim_{v\to +\infty} \left( (1+\varepsilon) \ln v
-\frac{2p_1}{b^2} \int_x^v u^{\alpha-2}\,du \right)\\
=&\lim_{v\to +\infty} \left( (1+\varepsilon) \ln v
-\frac{2p_1 \left( v^{\alpha-1}-x^{\alpha-1} \right)}{b^2 (\alpha-1)} \right)
= -\infty
\end{split}
\end{equation*}
for all \(\varepsilon >0\), Lemma~\ref{lem:A.1} gives \(I_1< +\infty\).

If \(p_2 >0\), then \(\mathbb{P} \bigl[ \lim_{t\to t_{(0,+\infty)}^*} X_t= +\infty \bigr] =1 \)
by Proposition~\ref{pr:2.1}.

For \(p_2 =0\), we have
\[
I_2
=-\int_0^x \exp \left\{ \frac{2p_1}{b^2} \int_x^v u^{\alpha-2}\,du \right\} dv
=-\int_0^x \exp \left\{ \frac{2p_1 \left( v^{\alpha-1}-x^{\alpha-1} \right)}{b^2 (\alpha-1)} \right\} dv
>-\infty.
\]
Hence, in this case \( \lim_{t\to t_{(0,+\infty)}^*} X_t\) exists a.s.,
\(0< \mathbb{P} \bigl[ \lim_{t\to t_{(0,+\infty)}^*} X_t= +\infty \bigr] <1 \),
and \(\mathbb{P} \bigl[ \lim_{t\to t_{(0,+\infty)}^*} X_t= 0 \bigr]
=1- \mathbb{P} \bigl[ \lim_{t\to t_{(0,+\infty)}^*} X_t= +\infty \bigr] \)
by Theorem~3.1 in~\cite[p.~447]{no11}.
\end{example}

\begin{example}
\label{ex:2.6}
Let
\begin{equation}
\label{eq:2.3}
p(u)=
\begin{cases}
p_2 u^2 +p_1 u +p_0 &\text{if} \quad u\ge0, \\
p_0 &\text{if} \quad u<0.
\end{cases}
\end{equation}
If \(p_0 \ge0\), \(p_1 \ge0\),  and \(p_2 >0\), then \(p(u)\) has all the properties required.

For all \(\varepsilon >0\), we have
\begin{equation*}
\begin{split}
&\lim_{v\to +\infty} \left( (1+\varepsilon) \ln v
-\frac{2}{b^2} \int_x^v \frac{p_2 u^2 +p_1 u +p_0}{u^2}\,du \right)
\le \lim_{v\to +\infty} \left( (1+\varepsilon) \ln v
-\frac{2}{b^2} \int_x^v p_2\,du \right)\\
=&\lim_{v\to +\infty} \left( (1+\varepsilon) \ln v
-\frac{2p_2 (v-x)}{b^2} \right)
= -\infty.
\end{split}
\end{equation*}
Hence, \(I_1< +\infty\) by Lemma~\ref{lem:A.1}.

If \(p_0 >0\), then \(\mathbb{P} \bigl[ \lim_{t\to t_{(0,+\infty)}^*} X_t= +\infty \bigr] =1 \)
by Proposition~\ref{pr:2.1}.

For \(p_0 =0\), we get
\[
I_2= -\int_0^x \exp\left\{ \frac{2}{b^2} \int_v^x \frac{p_2 u^2 +p_1 u}{u^2}\,du \right\} dv
=-\int_0^x \left(\frac{u}{y}\right)^{2p_1/b^2} \!\!\cdot
\exp \left\{ \frac{2p_2(x-v)}{b^2} \right\}dv.
\]
This gives \(I_2> -\infty\) for \(2p_1 < b^2\), and \(I_2= -\infty\) for \(2p_1\ge b^2\).
Consequently, if \(2p_1 < b^2\), then \( \lim_{t\to t_{(0,+\infty)}^*} X_t\) exists a.s.,
\(0< \mathbb{P} \bigl[ \lim_{t\to t_{(0,+\infty)}^*} X_t= +\infty \bigr] <1 \),
and \(\mathbb{P} \bigl[ \lim_{t\to t_{(0,+\infty)}^*} X_t= 0 \bigr]
=1- \mathbb{P} \bigl[ \lim_{t\to t_{(0,+\infty)}^*} X_t= +\infty \bigr] \);
if \(2p_1\ge b^2\), then \(\mathbb{P} \bigl[ \lim_{t\to t_{(0,+\infty)}^*} X_t= +\infty \bigr] =1 \).
\end{example}

One question still unanswered is whether \(t_{(0,+\infty)}^*\) is finite. We now study it
under the conditions of Example~\ref{ex:2.6}.

\begin{theorem}
\label{th:2.7}
Let \((X_t)_{t\ge0}\) be a strong solution of~\eqref{eq:2.1} and \(p(u)\)
be defined by~\eqref{eq:2.3} with \(p_0 \ge0\), \(p_1 \ge0\), and \(p_2 >0\). If
\(p_0 =0\) and \(\frac{2p_1}{b^2} <1\), then
\begin{equation}
\label{eq:2.4}
\mathbb{P} \bigl[ t_{(0,+\infty)}^*< \infty,\, X_{t_{(0,+\infty)}^*}\!= +\infty \bigr]
=\frac{\int_{0}^x v^{-2p_1/b^2} \!\cdot \exp \left\{ -\frac{2p_2 v}{b^2} \right\} dv}
{\int_{0}^{+\infty} v^{-2p_1/b^2} \!\cdot \exp \left\{ -\frac{2p_2 v}{b^2} \right\} dv};
\end{equation}
if either \(p_0 =0\) and \(\frac{2p_1}{b^2} \ge1\) or \(p_0 >0\), then
\begin{equation}
\label{eq:2.5}
\mathbb{P} \bigl[ t_{(0,+\infty)}^*< \infty,\, X_{t_{(0,+\infty)}^*}\!= +\infty \bigr] =1.
\end{equation}
\end{theorem}

\begin{proof}
Let \(n_0= \min\{n\in \mathbb{N} \colon 1/n<x\}\). For all integer \(n\) such that \(n\ge n_0\),
we denote by \(t_{(1/n,+\infty)}^*\) the first exit time from \((1/n,+\infty)\) for \((X_t)_{t\ge0}\), i.e.
\(t_{(1/n,+\infty)}^*= \inf\{t\ge0\colon X_t\notin (1/n,+\infty)\}\).

Note that the sequence of events \(\bigl(\bigl\{ \omega\in\Omega \colon t_{(1/n,+\infty)}^*(\omega)< \infty,\,
X_{t_{(1/n,+\infty)}^*}(\omega)= +\infty \bigr\}\bigr)_{n\ge n_0} \) is monotone nondecreasing.
Hence,
\begin{equation*}
\begin{split}
&\lim_{n\to\infty} \bigl\{ \omega\in\Omega \colon t_{(1/n,+\infty)}^*(\omega)< \infty,\,
X_{t_{(1/n,+\infty)}^*}(\omega)= +\infty \bigr\}\\
=&\bigcup_{n=n_0}^{\infty} \bigl\{ \omega\in\Omega \colon t_{(1/n,+\infty)}^*(\omega)< \infty,\,
X_{t_{(1/n,+\infty)}^*}(\omega)= +\infty \bigr\}.
\end{split}
\end{equation*}

Furthermore,
\begin{equation*}
\begin{split}
&\bigcup_{n=n_0}^{\infty} \bigl\{ \omega\in\Omega \colon t_{(1/n,+\infty)}^*(\omega)< \infty,\,
X_{t_{(1/n,+\infty)}^*}(\omega)= +\infty \bigr\}\\
=&\bigl\{ \omega\in\Omega \colon t_{(0,+\infty)}^*(\omega)< \infty,\,
X_{t_{(0,+\infty)}^*}(\omega)= +\infty \bigr\}.
\end{split}
\end{equation*}

Therefore by the continuity of probability measures, we conclude that
\begin{equation}
\label{eq:2.6}
\begin{split}
\mathbb{P} \bigl[ t_{(0,+\infty)}^*< \infty,\, X_{t_{(0,+\infty)}^*}\!= +\infty \bigr]
&=\mathbb{P} \Bigl[ \lim_{n\to\infty} \bigl\{ t_{(1/n,+\infty)}^*< \infty,\,
X_{t_{(1/n,+\infty)}^*}\!= +\infty \bigr\}\Bigr]\\
&=\lim_{n\to\infty}\mathbb{P} \bigl[ t_{(1/n,+\infty)}^*< \infty,\, X_{t_{(1/n,+\infty)}^*}\!= +\infty \bigr].
\end{split}
\end{equation}

From \cite[p.~343--344]{no13} it follows that
\( \mathbb{E} [t_{(1/n,+\infty)}^*] =M_n(x) \) for all \(n\ge n_0\),
where \(M_n(x)\) is a solution of the boundary value problem
\begin{equation}
\label{eq:2.7}
\frac{1}{2} b^2 x^2 M_n''(x) +(p_2 x^2 +p_1 x +p_0) M_n'(x) =-1,
\quad M_n \left( \frac{1}{n} \right)=0, \: M_n(+\infty)=0,
\end{equation}
which can be solved in a certain way (see, e.g.,~\cite{no1}).
Here and subsequently, the value of a function at \(+\infty\) stands for its limit
as the value of the argument tends to \(+\infty\).

Boundary value problem~\eqref{eq:2.7} has the unique solution
\[
M_n(x)= \frac{2m_n(x)}{b^2 m_n(+\infty)}
\int_{1/n}^{+\infty}\frac{m_n(+\infty)-m_n(z)}{z^2 m_n'(z)} \,dz-
\frac{2}{b^2} \int_{1/n}^{x}\frac{m_n(x)-m_n(z)}{z^2 m_n'(z)} \,dz,
\]
where
\[
m_n(x)=\int_{1/n}^{x} \exp\left\{ -\frac{2}{b^2}
\int_{1/n}^{v} \frac{p_2 u^2 +p_1 u +p_0}{u^2} \,du \right\}dv.
\]

Note that \(m_n(+\infty) <+\infty\). Furthermore, since
\[
\lim_{z\to+\infty} \frac{m_n(+\infty)-m_n(z)}{m_n'(z)}
=\lim_{z\to+\infty} \frac{\int_{z}^{+\infty} \exp\left\{ -\frac{2}{b^2} \int_{1/n}^{v}
\frac{p_2 u^2 +p_1 u +p_0}{u^2} \,du \right\}dv}
{\exp\left\{ -\frac{2}{b^2} \int_{1/n}^{z}\frac{p_2 u^2 +p_1 u +p_0}{u^2} \,du \right\}}
=\frac{b^2}{2p_2} <+\infty
\]
(here we applied L'Hopital's rule) and
\(\int_{1/n}^{+\infty}\frac{1}{z^2} \,dz <+\infty\), we get
\[
\int_{1/n}^{+\infty}\frac{m_n(+\infty)-m_n(z)}{z^2 m_n'(z)} \,dz <+\infty.
\]

Thus, \(\mathbb{E} [t_{(1/n,+\infty)}^*]< \infty\) for all \(n\ge n_0\).
This gives \(\mathbb{P} [t_{(1/n,+\infty)}^*< \infty]=1\) for all \(n\ge n_0\).
Moreover, by~\cite[p.~343--344]{no13}, we have
\begin{equation}
\label{eq:2.8}
\begin{split}
\mathbb{P} \bigl[ X_{t_{(1/n,+\infty)}^*}\!= +\infty \bigr]
&=\frac{\int_{1/n}^x \exp \left\{ -\frac{2}{b^2} \int_{1/n}^v \frac{p_2 u^2 +p_1 u +p_0}{u^2} \,du \right\} dv}
{\int_{1/n}^{+\infty} \exp \left\{ -\frac{2}{b^2} \int_{1/n}^v \frac{p_2 u^2 +p_1 u +p_0}{u^2} \,du \right\} dv}\\
&=\frac{\int_{1/n}^x v^{-2p_1/b^2} \!\cdot \exp \left\{ \frac{2p_0}{b^2 v} -\frac{2p_2 v}{b^2} \right\} dv}
{\int_{1/n}^{+\infty} v^{-2p_1/b^2} \!\cdot \exp \left\{ \frac{2p_0}{b^2 v} -\frac{2p_2 v}{b^2} \right\} dv}.
\end{split}
\end{equation}

Consequently,~\eqref{eq:2.6} and~\eqref{eq:2.8} yield
\begin{equation}
\label{eq:2.9}
\mathbb{P} \bigl[ t_{(0,+\infty)}^*< \infty,\, X_{t_{(0,+\infty)}^*}\!= +\infty \bigr]
=\lim_{n\to\infty}
\frac{\int_{1/n}^x v^{-2p_1/b^2} \!\cdot \exp \left\{ \frac{2p_0}{b^2 v} -\frac{2p_2 v}{b^2} \right\} dv}
{\int_{1/n}^{+\infty} v^{-2p_1/b^2} \!\cdot \exp \left\{ \frac{2p_0}{b^2 v} -\frac{2p_2 v}{b^2} \right\} dv}.
\end{equation}

Consider now two cases.
\begin{enumerate}
\item
If \(p_0 =0\) and \(\frac{2p_1}{b^2} <1\), then both of the integrals in the right-hand side of~\eqref{eq:2.9}
are finite as \(n\to\infty\). This yields~\eqref{eq:2.4}.
Note that in this case
\(0< \mathbb{P} \bigl[ t_{(0,+\infty)}^*< \infty,\, X_{t_{(0,+\infty)}^*}\!= +\infty \bigr] <1\).

\item
If either \(p_0 =0\) and \(\frac{2p_1}{b^2} \ge1\) or \(p_0 >0\), then both of the integrals
in the right-hand side of~\eqref{eq:2.9} are infinite as \(n\to\infty\).
Applying L'Hopital's rule we obtain~\eqref{eq:2.5}.
\end{enumerate}

The theorem is proved.
\end{proof}

\begin{remark}
\label{rem:2.8}
Since \(c(u)\) is positive by our assumption, the surplus of the insurance company
becomes infinitely large in finite time a.s. if the premium intensity
is a quadratic function and the claims do not arrive.
Note that the time interval between two successive claims can be large enough
with positive probability. Hence, the process \(\bigl(X_t(x)\bigr)_{t\ge0}\)
that follows~\eqref{eq:1.3} goes to \(+\infty\) with positive probability.
It is clear that the ruin does not occur in this case. Consequently, from now on
we can consider \(\bigl(X_t(x)\bigr)_{t\ge0}\) up to the minimum from the ruin time
and its possible explosion.
\end{remark}

\section{Existence and uniqueness theorem}\label{se:3}

Consider now equation~\eqref{eq:1.3}. Let \(t^*(x)\) be a possible explosion time of
\(\bigl(X_t(x)\bigr)_{t\ge0}\), i.e. \(t^*(x)= \inf\{t\ge0\colon X_t(x)=\infty\}\).
To shorten notation, we let \(t^*\) stand for \(t^*(x)\).

\begin{theorem}
\label{th:3.1}
If \(c(u)\) is a locally Lipschitz continuous function on \(\mathbb{R}\), then~\eqref{eq:1.3}
has a unique strong solution up to the time \(\tau \wedge t^*\).
\end{theorem}

\begin{proof}
Since the process \((N_t)_{t\ge 0}\) is homogeneous, it has only a finite number of jumps
on any finite time interval a.s. To prove the theorem, we study~\eqref{eq:1.3} between
two successive jumps of \(N_t\).

Let us first consider~\eqref{eq:1.3} on the time interval \([\tau_0, \tau_1)\). It can be rewritten as
\begin{equation}
\label{eq:3.1}
X_t= X_{\tau_0}+ \int_{\tau_0}^t \bigl( c(X_s)+aX_s \bigr) \,ds+ b\int_{\tau_0}^t X_s \,dW_s,
\quad
\tau_0 \le t < \tau_1.
\end{equation}

By Theorem~3.1 in~\cite[p.~178--179]{no11}, the locally Lipschitz continuity of
\(c(u)+au\) and \(bu\) on \(\mathbb{R}\) implies
the existence of a unique strong solution of~\eqref{eq:3.1} on \([\tau_0, \tau_1 \wedge t^*)\).
Moreover, the comparison theorem (see, e.g., Theorem~1.1 in~\cite[p.~437--438]{no11})
shows that this solution is not less then the solution of
\begin{equation}
\label{eq:3.2}
X_t= X_{\tau_0}+ a\int_{\tau_0}^t X_s \,ds+ b\int_{\tau_0}^t X_s \,dW_s,
\quad
\tau_0 \le t \le \tau_1 \wedge t^*,
\end{equation}
a.s. Since the solution of~\eqref{eq:3.2} is positive, so is the solution of~\eqref{eq:3.1}
on \([\tau_0, \tau_1 \wedge t^*)\).
Hence, \(X_{t^*}= +\infty\) if \(t^* \le \tau_1\). Thus, the ruin does not occur
up to the time \(\tau_1 \wedge t^*\).

If \(t^* \le \tau_1\), then the theorem follows. Otherwise \(X_{\tau_{1-}}< +\infty\) and we set
\(X_{\tau_1}= X_{\tau_{1-}} -Y_1\). Next, if \(X_{\tau_1}<0\), then \(\tau=\tau_1\), which
completes the proof. Otherwise we consider~\eqref{eq:1.3} on the time interval \([\tau_1, \tau_2)\).
We rewrite it as
\begin{equation}
\label{eq:3.3}
X_t= X_{\tau_1}+ \int_{\tau_1}^t \bigl( c(X_s)+aX_s \bigr) \,ds+ b\int_{\tau_1}^t X_s \,dW_s,
\quad
\tau_1 \le t < \tau_2.
\end{equation}

Repeating the same arguments, we conclude that~\eqref{eq:3.3} has a unique strong solution
on \([\tau_1, \tau_2 \wedge t^*)\) and the ruin does not occur up to the time \(\tau_2 \wedge t^*\).

Thus, we have proved that~\eqref{eq:1.3} has a unique strong solution on \([0, \tau_2 \wedge t^*)\),
which is our assertion if \(t^* \le \tau_2\). For the case \(t^* > \tau_2\), we set
\(X_{\tau_2}= X_{\tau_{2-}} -Y_2\). Next, if \(X_{\tau_2} <0\), then \(\tau=\tau_2\),
which proves the theorem. Otherwise we continue in this fashion and prove the theorem by induction.
\end{proof}

\begin{remark}
\label{rem:3.2}
Note that if \(t^* <\infty\), then the proof of Theorem~\ref{th:3.1} implies \(X_{t^*}= +\infty\)
and~\eqref{eq:1.3} also holds for \(t=t^*\) provided that we let both of its sides be equal to \(+\infty\).
In addition, if \(\tau <\infty\), then we set \(X_{\tau}= X_{\tau_{i-}}-Y_i\), where \(i\) is the number
of the claim that caused the ruin, and~\eqref{eq:1.3} also holds for \(t=\tau\).
\end{remark}

\section{Supermartingale property for the exponential process}\label{se:4}

Let the stopped process \(\bigl(\tilde X_t(x)\bigr)_{t\ge 0}\) be defined by
\(\tilde X_t(x)= X_{t \wedge \tau \wedge t^*}(x)\).
Note that \(\bigl(\tilde X_t(x)\bigr)_{t\ge 0}\) is a solution of~\eqref{eq:1.3}
provided that so is \(\bigl(X_t(x)\bigr)_{0\le t< \tau \wedge t^*}\).

For all \(r\ge 0\), we define the processes \(\bigl(U_t(x,r)\bigr)_{t\ge 0}\)
and \(\bigl(V_t(x,r)\bigr)_{t\ge 0}\) by
\[
U_t(x,r)= -r \tilde X_t(x) \quad \text{and} \quad V_t(x,r)= e^{U_t(x,r)}.
\]

In what follows, we write \(\tilde X_t\), \(U_t\), and \(V_t\) instead of \(\tilde X_t(x)\), \(U_t(x,r)\), and \(V_t(x,r)\), respectively, when no confusion can arise.

\begin{theorem}
\label{th:4.1}
If~\eqref{eq:1.3} has a unique strong solution up to the time \(\tau \wedge t^*\) and
there exists \(\hat r \in (0,r_{\infty})\) such that
\begin{equation}
\label{eq:4.1}
\frac{\hat r^2 b^2}{2} u^2 -\hat r \bigl( c(u)+au \bigr) +\lambda h(\hat r) \le 0
\quad \text{for all}
\quad u\ge 0,
\end{equation}
then \(\bigl(V_t(x,r)\bigr)_{t\ge 0}\) is an \((\mathfrak{F}_t)\)-supermartingale.
\end{theorem}

\begin{proof}
Since \(\bigl(\tilde X_t\bigr)_{t\ge 0}\) is a solution of~\eqref{eq:1.3}, we have
\begin{equation}
\label{eq:4.2}
U_t= -rx -r\int_0^{t \wedge \tau \wedge t^*} \!\bigl( c(X_s)+aX_s \bigr) \,ds
-rb\int_0^{t \wedge \tau \wedge t^*} \!X_s \,dW_s
+r\! \sum_{i=1}^{N_{t \wedge \tau \wedge t^*}} \!Y_i,
\quad
t\ge 0.
\end{equation}

The process \((\tilde X_t)_{t\ge 0}\) is a sum of local martingales and c{\`a}dl{\`a}g
processes of locally bounded variation.\\
Indeed, since
\(\mathbb{E} \left[\left| \int_0^{t \wedge \tau \wedge t^* \wedge T_n}
\!X_s \,dW_s \right|\right] <+\infty \)
for all \(t\ge 0\), the process \( \left(\int_0^{t \wedge \tau \wedge t^*} \!X_s \,dW_s \right)_{t\ge 0}\)
is a local \((\mathfrak{F}_t)\)-martingale with the localizing sequence \((T_n)_{n\ge 1}\), where
\[
T_n= \inf\{t\ge0\colon X_t\ge n\} \wedge n.
\]
Similarly, \( \left(\int_0^{t \wedge \tau \wedge t^*} \!X_s \,ds \right)_{t\ge 0}\)
and \( \left(\int_0^{t \wedge \tau \wedge t^*} \!c(X_s) \,ds \right)_{t\ge 0}\)
are c{\`a}dl{\`a}g processes of locally bounded variation
with the localizing sequence \((T_n)_{n\ge 1}\).
Next, the process
\[
\left( \sum_{i=1}^{N_{t \wedge \tau \wedge t^*}} Y_i
-\lambda\mu (t \wedge \tau \wedge t^*) \right)_{t\ge 0}
\]
is a compensated process with independent increments. Hence, it is an \((\mathfrak{F}_t)\)-martingale.

Thus, \((U_t)_{t\ge 0}\) is an \((\mathfrak{F}_t)\)-semimartingale and so is \((V_t)_{t\ge 0}\).
Applying It{\^o}'s formula
\begin{equation*}
\begin{split}
g(U_t)- g(U_0)
=\int_{0_+}^{t} g'(U_{s_-}) \,dU_s
&+\frac{1}{2} \int_{0_+}^{t} g''(U_{s_-}) \,d \langle U^c,U^c \rangle_s\\
&+\sum_{0<s\le t}\!\bigl( g(U_s)- g(U_{s_-})- g'(U_{s_-})(U_s- U_{s_-}) \bigr),
\quad
t\ge 0,
\end{split}
\end{equation*}
where \((U_t)_{t\ge 0}\) is a semimartingale, \((U_t^c)_{t\ge 0}\) is a continuous component of the local martingale in the decomposition of \((U_t)_{t\ge 0}\), and \(g\in C^2(\mathbb{R})\), we get
\begin{equation}
\label{eq:4.3}
\begin{split}
V_t= e^{-rx}
+\int_{0_+}^{t \wedge \tau \wedge t^*} e^{U_{s_-}} \,dU_s
&+\frac{1}{2} \int_{0_+}^{t \wedge \tau \wedge t^*} e^{U_{s_-}} \,d \langle U^c,U^c \rangle_s\\
&+\sum_{0<s\le t \wedge \tau \wedge t^*}\!\bigl( e^{U_s}- e^{U_{s_-}}- e^{U_{s_-}}(U_s- U_{s_-}) \bigr),
\quad
t\ge 0,
\end{split}
\end{equation}
where
\[
U_{t_-}= -rx -r\int_0^{t \wedge \tau \wedge t^*} \!\bigl( c(X_s)+aX_s \bigr) \,ds
-rb\int_0^{t \wedge \tau \wedge t^*} \!X_s \,dW_s
+r\! \sum_{0< s\le t_- \wedge \tau \wedge t^*} \!Y_{N_s} \mathbb{I}_{\{ \Delta N_s \ne 0 \}},
\]
\[
dU_s= -r \bigl( c(\tilde X_s)+ a\tilde X_s \bigr) \,ds-
rb \tilde X_s \,dW_s+ rY_{N_s} \mathbb{I}_{\{ \Delta N_s \ne 0 \}},
\]
\[
d \langle U^c,U^c \rangle_s= r^2 b^2 \tilde X_s^2,
\]
\[
e^{U_s}- e^{U_{s_-}}= e^{U_{s_-}}(e^{rY_{N_s} \mathbb{I}_{\{ \Delta N_s \ne 0 \}}} -1),
\]
\[
U_s- U_{s_-}= rY_{N_s} \mathbb{I}_{\{ \Delta N_s \ne 0 \}},
\]
\[
\Delta N_s= N_s- N_{s_-}.
\]

Substituting all the above equalities into~\eqref{eq:4.3} yields
\begin{equation}
\label{eq:4.4}
\begin{split}
V_t= e^{-rx}
&-r\int_{0_+}^{t \wedge \tau \wedge t^*} e^{U_{s_-}} \bigl( c(X_s)+aX_s \bigr) \,ds
-rb\int_{0_+}^{t \wedge \tau \wedge t^*} X_s \,dW_s\\
&+r\! \sum_{0< s\le t \wedge \tau \wedge t^*} \!e^{U_s} Y_{N_s} \mathbb{I}_{\{ \Delta N_s \ne 0 \}}
+\frac{1}{2} r^2 b^2 \int_{0_+}^{t \wedge \tau \wedge t^*} e^{U_{s_-}} X_s^2 \,ds\\
&+\sum_{0< s\le t \wedge \tau \wedge t^*} \!e^{U_s}
\bigl( e^{rY_{N_s} \mathbb{I}_{\{ \Delta N_s \ne 0 \}}}-
1- rY_{N_s} \mathbb{I}_{\{ \Delta N_s \ne 0 \}} \bigr),
\quad
t\ge 0.
\end{split}
\end{equation}

Simplifying~\eqref{eq:4.4} gives
\begin{equation}
\label{eq:4.5}
\begin{split}
V_t= e^{-rx}
&+\int_{0_+}^{t \wedge \tau \wedge t^*} e^{U_{s_-}}
\left( \frac{1}{2} r^2 b^2 X_s^2 -r\bigl( c(X_s)+aX_s \bigr) \right) ds\\
&-rb\int_{0_+}^{t \wedge \tau \wedge t^*} X_s \,dW_s
+\sum_{0< s\le t \wedge \tau \wedge t^*} \!e^{U_s}
\bigl( e^{rY_{N_s} \mathbb{I}_{\{ \Delta N_s \ne 0 \}}}-1 \bigr),
\quad
t\ge 0.
\end{split}
\end{equation}

Next, the process
\[
\left(  \sum_{0< s\le t \wedge \tau \wedge t^*} \!e^{U_s}
\bigl( e^{rY_{N_s} \mathbb{I}_{\{ \Delta N_s \ne 0 \}}}-1 \bigr) \right)
_{t\ge 0}
\]
is nondecreasing and can be written in the integral form
\[
\sum_{0< s\le t \wedge \tau \wedge t^*} \!e^{U_s}
\bigl( e^{rY_{N_s} \mathbb{I}_{\{ \Delta N_s \ne 0 \}}}-1 \bigr)=
\int_{0_+}^{t \wedge \tau \wedge t^*} e^{U_{s_-}} \,dQ_s,
\quad
t\ge 0,
\]
where
\[
Q_t= \sum_{0< s\le t \wedge \tau \wedge t^*} \!
\bigl( e^{rY_{N_s} \mathbb{I}_{\{ \Delta N_s \ne 0 \}}}-1 \bigr).
\]
By Wald's identity, \(\mathbb{E} [Q_t]= \lambda th(r)\).
Hence, \(\mathbb{E} [Q_t]< +\infty\) for all \(t\ge 0\) and \(r< r_{\infty}\).
Furthermore, since \( (Q_t)_{t\ge 0} \) is a process with independent increments,
the compensated process \(\bigl(Q_t- \mathbb{E}[Q_t]\bigr)_{t\ge 0} \)
is an \((\mathfrak{F}_t)\)-martingale.
Thus,
\[
\left( \sum_{0< s\le t \wedge \tau \wedge t^*} \!
\bigl( e^{rY_{N_s} \mathbb{I}_{\{ \Delta N_s \ne 0 \}}}-1 \bigr)-
\lambda h(r) \int_{0_+}^{t \wedge \tau \wedge t^*} e^{U_{s_-}} \,ds \right)
_{t\ge 0}
\]
is a local \((\mathfrak{F}_t)\)-martingale with the localizing sequence \((T_n)_{n\ge 1}\).

Since \(\Bigl( -rb\int_{0_+}^{t \wedge \tau \wedge t^*} X_s \,dW_s \Bigr)_{t\ge 0} \)
is also a local \((\mathfrak{F}_t)\)-martingale with the localizing sequence \((T_n)_{n\ge 1}\), so is
\[
\left( -rb\int_{0_+}^{t \wedge \tau \wedge t^*} X_s \,dW_s+
\sum_{0< s\le t \wedge \tau \wedge t^*} \!
\bigl( e^{rY_{N_s} \mathbb{I}_{\{ \Delta N_s \ne 0 \}}}-1 \bigr)-
\lambda h(r) \int_{0_+}^{t \wedge \tau \wedge t^*} e^{U_{s_-}} \,ds \right)
_{t\ge 0}.
\]

We define the process \((R_t)_{t\ge 0}\) by
\begin{equation*}
\begin{split}
R_t= V_t&- V_0+ rb\int_{0_+}^{t \wedge \tau \wedge t^*} X_s \,dW_s\\
&-\sum_{0< s\le t \wedge \tau \wedge t^*} \!
\bigl( e^{rY_{N_s} \mathbb{I}_{\{ \Delta N_s \ne 0 \}}}-1 \bigr)+
\lambda h(r) \int_{0_+}^{t \wedge \tau \wedge t^*} e^{U_{s_-}} \,ds.
\end{split}
\end{equation*}

Substituting \(V_t\) from~\eqref{eq:4.5} we obtain
\[
R_t=
\int_{0_+}^{t \wedge \tau \wedge t^*} e^{-rX_{s_-}}
\left( \frac{1}{2} r^2 b^2 X_s^2 -r\bigl( c(X_s)+aX_s \bigr) +\lambda h(r) \right) ds,
\quad
t\ge 0.
\]

By the Doob-Meyer decomposition, \((V_t)_{t\ge 0}\) is a local \((\mathfrak{F}_t)\)-supermartingale
with the localizing sequence \((T_n)_{n\ge 1}\) provided that \((R_t)_{t\ge 0}\) is
a measurable nonincreasing process, i.e.
\begin{equation}
\label{eq:4.6}
\int_{t_1 \wedge \tau \wedge t^*}^{t_2 \wedge \tau \wedge t^*} e^{-rX_{s_-}}
\left( \frac{1}{2} r^2 b^2 X_s^2 -r\bigl( c(X_s)+aX_s \bigr) +\lambda h(r) \right) ds \le 0 \quad
\text{for all} \quad
t_2 \ge t_1 \ge 0.
\end{equation}

By the assumption of the theorem, there exists \(\hat r \in (0,r_{\infty})\)
such that~\eqref{eq:4.1} holds. Therefore,~\eqref{eq:4.6} is true with \(r= \hat r\)
and \((V_t(x,\hat r))_{t\ge 0}\) is a nonnegative local \((\mathfrak{F}_t)\)-supermartingale
with the localizing sequence \((T_n)_{n\ge 1}\).

By Fatou's lemma, for all \(t_2 \ge t_1 \ge 0\), we get
\begin{equation*}
\begin{split}
0&\le \mathbb{E} \bigl[ V_{t_2}(x,\hat r) \,/\, \mathfrak{F}_{t_1} \bigr]
=\mathbb{E} \Bigl[ \lim_{n\to\infty}V_{t_2 \wedge T_n}(x,\hat r) \,/\, \mathfrak{F}_{t_1} \Bigr]
=\mathbb{E} \Bigl[ \liminf_{n\to\infty}V_{t_2 \wedge T_n}(x,\hat r) \,/\, \mathfrak{F}_{t_1} \Bigr]\\
&\le \liminf_{n\to\infty} \mathbb{E} \bigl[ V_{t_2 \wedge T_n}(x,\hat r) \,/\, \mathfrak{F}_{t_1} \bigr]
\le \liminf_{n\to\infty} V_{t_1 \wedge T_n}(x,\hat r)
=V_{t_1}(x,\hat r).
\end{split}
\end{equation*}

Hence, \((V_t(x,\hat r))_{t\ge 0}\) is an \((\mathfrak{F}_t)\)-supermartingale,
which completes the proof.
\end{proof}

Theorem~\ref{th:4.1} allows us to get an exponential bound for the ruin probability
under certain conditions.

\section{Exponential bound for the ruin probability}\label{se:5}

Let the premium intensity \(c(u)\) be a quadratic function for \(u\ge0\), i.e.
\begin{equation}
\label{eq:5.1}
c(u)=
\begin{cases}
c_2 u^2 +c_1 u +c_0 &\text{if} \quad u\ge0, \\
c_0 &\text{if} \quad u<0,
\end{cases}
\end{equation}
where \(c_2 \ne 0\). The function \(c(u)\) is strictly increasing and positive on \([0,+\infty)\)
if and only if \(c_0 >0\), \(c_1 \ge0\), and \(c_2 >0\).
This model implies that the premium intensity grows rapidly with increasing surplus.

\begin{theorem}
\label{th:5.1}
Let the surplus process \( \left(X_t(x)\right)_{t\ge 0} \) follow~\eqref{eq:1.3}
under the above assumptions, the premium intensity \(c(u)\) be defined by~\eqref{eq:5.1}
with \(c_0 >0\), \(c_1 \ge0\), and \(c_2 >0\),
and at least one of the following two conditions holds

1) \(\frac{2c_2}{b^2}< r_{\infty}\) and
\( h \left( \frac{2c_2}{b^2} \right) \le \frac{2 c_0 c_2}{b^2 \lambda} \);

2) \(\lambda \mu< c_0\).\\
Then for all \(x\ge 0\), we have
\begin{equation}
\label{eq:5.2}
\psi(x) \le e^{-\hat r x},
\end{equation}
where \(\hat r= \frac{2c_2}{b^2}\) if condition~1) holds,
and \(\hat r= \min \left\{r_0, \frac{2c_2}{b^2} \right\}\) if condition~2) holds.
Here \(r_0\) stands for a unique positive solution of
\begin{equation}
\label{eq:5.3}
h(r)=\frac{c_0 r}{\lambda}.
\end{equation}
\end{theorem}

\begin{proof}
Since \(c(u)\) defined by~\eqref{eq:5.1} is a locally Lipschitz continuous function
on \(\mathbb{R}\), equation~\eqref{eq:1.3} has a unique strong solution up to the time
\(\tau \wedge t^*\) by Theorem~\ref{th:3.1}. According to Theorem~\ref{eq:4.1},
if there exists \(\hat r \in (0,r_{\infty})\) such that
\begin{equation}
\label{eq:5.4}
\left( \frac{\hat r^2 b^2}{2}- \hat r c_2 \right) u^2 -\hat r (a+c_1) u -\hat r c_0 +\lambda h(\hat r) \le 0
\quad \text{for all}
\quad u\ge 0,
\end{equation}
then \(\bigl(V_t(x,\hat r)\bigr)_{t\ge 0}\) is an \((\mathfrak{F}_t)\)-supermartingale.

Condition~\eqref{eq:5.4} holds in one of the two following cases.
\begin{enumerate}

\item
The coefficient of \(u^2\) is equal to \(0\), i.e. \(\hat r= \frac{2c_2}{b^2}\).

Then~\eqref{eq:5.4} is true if and only if
\[
\frac{2c_2}{b^2}< r_{\infty}
\quad \text{and} \quad
-\frac{2c_2}{b^2} c_0+ \lambda h \left( \frac{2c_2}{b^2} \right) \le 0,
\]
which coincides with condition~1) of the theorem.

\item
The coefficient of \(u^2\) is negative, i.e. \(\hat r \in \left( 0,\frac{2c_2}{b^2} \right)\).

Since \(u= \frac{a+c_1}{\hat r b^2- 2c_2}\), which is negative, maximizes the left-hand side
of~\eqref{eq:5.4}, the last one is true if and only if
\begin{equation}
\label{eq:5.5}
\hat r \in \left( 0, \min \left\{ \frac{2c_2}{b^2}, r_{\infty} \right\} \right)
\end{equation}
and
\begin{equation}
\label{eq:5.6}
\lambda h(\hat r) \le c_0 \hat r.
\end{equation}

Consider the functions \(g_1(r)= \lambda h(r)\) and \(g_2(r)= c_0 r\) on \([0,r_{\infty})\).
Note that \(g_1(0)=0\), \(g_2(0)=0\), \(g'_1(0)= \lambda \mu\), and \(g'_2(0)=c_0\).
On account of the properties of \(h(r)\), this gives us the following.

If \(\lambda \mu \ge c_0\), then \(g_2(r)< g_1(r)\) for all \(r\in (0,r_{\infty})\).
Hence, for no \(\hat r \in (0,r_{\infty})\) does~\eqref{eq:5.4} hold.

If \(\lambda \mu < c_0\), then the equation \(g_1(r) = g_2(r)\) has a unique
solution \(r_0 \in (0,r_{\infty})\). Therefore,~\eqref{eq:5.4} has a unique positive solution
and~\eqref{eq:5.6} is true for all \(\hat r \in (0,r_0]\).
Moreover,~\eqref{eq:5.5} must be satisfied. Consequently,~\eqref{eq:5.4} holds
for all \(\hat r \in (0,r_0]\) if \(r_0< \frac{2c_2}{b^2}\),
and for all \(\hat r \in \left( 0,\frac{2c_2}{b^2} \right)\) if \(r_0 \ge \frac{2c_2}{b^2}\).
\end{enumerate}

Thus, we have found out when \(\bigl(V_t(x,\hat r)\bigr)_{t\ge 0}\) is an \((\mathfrak{F}_t)\)-supermartingale.

Next, if \(\bigl(V_t(x,\hat r)\bigr)_{t\ge 0}\) is an \((\mathfrak{F}_t)\)-supermartingale,
then for all \(t\ge 0\), we get
\begin{equation}
\label{eq:5.7}
\begin{split}
e^{-\hat r x}&=
V_0(x,\hat r)\ge
\mathbb{E} \bigl[ V_t(x,\hat r) \,/\, \mathfrak{F}_0 \bigr]=
\mathbb{E} \bigl[ e^{-\hat r X_{t \wedge \tau \wedge t^*}(x)} \bigr]\\
&=\mathbb{E} \bigl[ e^{-\hat r X_{\tau}(x)} \cdot \mathbb{I}_{\{\tau(x)< t \wedge t^*\}} \bigr]+
\mathbb{E} \bigl[ e^{-\hat r X_{t \wedge t^*}(x)} \cdot \mathbb{I}_{\{\tau(x)\ge t \wedge t^*\}} \bigr]\\
&\ge \mathbb{E} \bigl[ e^{-\hat r X_{\tau}(x)} \cdot \mathbb{I}_{\{\tau(x)< t \wedge t^*\}} \bigr],
\end{split}
\end{equation}
where \(\hat r= \frac{2c_2}{b^2}\) if condition~1) of the theorem holds, and \(\hat r\)
is an arbitrary number from \((0,r_0]\) for \(r_0< \frac{2c_2}{b^2}\)
or from \(\left( 0,\frac{2c_2}{b^2} \right)\) for \(r_0\ge \frac{2c_2}{b^2}\)
if condition~2) of the theorem holds.

Letting \(t\to\infty\) in~\eqref{eq:5.7} gives
\begin{equation}
\label{eq:5.8}
\mathbb{E} \bigl[ e^{-\hat r X_{\tau}(x)} \cdot \mathbb{I}_{\{\tau(x)< t^*\}} \bigr] \le
e^{-\hat r x}.
\end{equation}

Since the surplus becomes infinitely large at the explosion time, the ruin does not occur after \(t^*\).
Hence,
\[
\bigl\{ \omega\in\Omega \colon \tau(x,\omega)< t^*(\omega) \bigr\}=
\bigl\{ \omega\in\Omega \colon \tau(x,\omega)< \infty \bigr\}
\]
and~\eqref{eq:5.8} can be rewritten as
\begin{equation}
\label{eq:5.9}
\mathbb{E} \bigl[ e^{-\hat r X_{\tau}(x)} \cdot \mathbb{I}_{\{\tau(x)< \infty\}} \bigr] \le
e^{-\hat r x}.
\end{equation}

Furthermore,
\[
\mathbb{E} \bigl[ e^{-\hat r X_{\tau}(x)} \cdot \mathbb{I}_{\{\tau(x)< \infty\}} \bigr] =
\mathbb{E} \bigl[ e^{-\hat r X_{\tau}(x)} \,/\, \tau(x)< \infty \bigr] \cdot \mathbb{P} [\tau(x)< \infty],
\]
and
\[
\mathbb{E} \bigl[ e^{-\hat r X_{\tau}(x)} \cdot \mathbb{I}_{\{\tau(x)< \infty\}} \bigr] \ge 1
\]
by the definition of the ruin time. Therefore, from~\eqref{eq:5.9} we conclude that
\[
\mathbb{P} [\tau(x)< \infty] \le \frac{e^{-\hat r x}}
{\mathbb{E} \bigl[ e^{-\hat r X_{\tau}(x)} \,/\, \tau(x)< \infty \bigr]} \le e^{-\hat r x},
\]
which yields~\eqref{eq:5.2}.

What is left is to note that the larger \(\hat r\) we choose, the better bound in~\eqref{eq:5.2} we get.
Thus, if condition~2) of the theorem holds and \(r_0< \frac{2c_2}{b^2}\), then we set \(\hat r= r_0\).
If condition~2) of the theorem holds and \(r_0 \ge \frac{2c_2}{b^2}\), then~\eqref{eq:5.2} is true for all
\(\hat r \in \left( 0,\frac{2c_2}{b^2} \right)\); hence, it is also true for \(\hat r= \frac{2c_2}{b^2}\).
This completes the proof.
\end{proof}

\appendix
\section{Sufficient conditions for finiteness of \(I_1\) and \(I_2\)}\label{se:A}

Consider now equation~\eqref{eq:2.1}. Let \(I_1\) and \(I_2\) be defined by~\eqref{eq:2.2}.
The following lemmas provide sufficient conditions for \(I_1\) and \(I_2\) being finite.

\begin{lemma}
\label{lem:A.1}
If
\begin{equation}
\label{eq:A.1}
\lim_{v\to +\infty} \left( (1+\varepsilon) \ln v
-\frac{2}{b^2} \int_x^v \frac{p(u)}{u^2}\,du \right)< +\infty
\quad\text{for some}\quad
\varepsilon>0,
\end{equation}
then \(I_1< +\infty\).
\end{lemma}

\begin{proof}
Since \(\int_x^{+\infty} \frac{1}{v^{1+\varepsilon}}\,dv <+\infty\) for all \(\varepsilon>0\),
it suffices to show that
\begin{equation}
\label{eq:A.2}
\lim_{v\to +\infty} \frac{\exp \left\{ -\frac{2}{b^2} \int_x^v \frac{p(u)}{u^2}\,du \right\}}
{\exp \left\{ -(1+\varepsilon)\ln v \right\}}< +\infty
\quad\text{for some}\quad
\varepsilon>0
\end{equation}
in order to get \(I_1< +\infty\).

We can rewrite~\eqref{eq:A.2} as
\[
\lim_{v\to +\infty} \exp \left\{ (1+\varepsilon)\ln v
-\frac{2}{b^2} \int_x^v \frac{p(u)}{u^2}\,du \right\}< +\infty
\quad\text{for some}\quad
\varepsilon>0,
\]
which gives~\eqref{eq:A.1}.
\end{proof}

\begin{lemma}
\label{lem:A.2}
If \(p(0)>0\), then \(I_2= -\infty\).
\end{lemma}

\begin{proof}
It is easily seen that
\begin{equation*}
\begin{split}
-I_2 &\ge \int_0^x \exp\left\{ \frac{2p(0)}{b^2} \int_v^x \frac{1}{u^2} \,du \right\} \,dv
=\exp \left\{ -\frac{2p(0)}{b^2 x} \right\} \cdot
\int_0^x \exp \left\{ \frac{2p(0)}{b^2 v} \right\} \,dv \\
&=\exp \left\{ -\frac{2p(0)}{b^2 x} \right\} \cdot
\int_{1/x}^{+\infty} \frac{1}{u^2} \exp \left\{ \frac{2p(0)u}{b^2} \right\} du =+\infty,
\end{split}
\end{equation*}
which proves the lemma.
\end{proof}

\end{document}